\documentclass[a4paper,10pt]{amsart}
\usepackage{amssymb, amsmath, graphicx, amsthm, enumitem, multicol,amsfonts}
\usepackage[margin=1in]{geometry}
\usepackage{cite}
\usepackage{ bbold }
\usepackage{longtable}
\usepackage{booktabs}
\usepackage{placeins}
\usepackage{hyperref}
\usepackage{tikz-cd}

\usepackage{pdfpages} 

\usepackage{cleveref} 
\usepackage{tikz}

\usepackage{graphicx} 
\usepackage{pgfplots} 
\usepackage{float} 
\pgfplotsset{width=10cm,compat=1.9}

\usepackage[mathscr]{euscript} 

\usepackage[utf8]{inputenc}

\DeclareMathOperator{\Z}{\mathbb{Z}}
\DeclareMathOperator{\N}{\mathbb{N}}

\newcommand{\cK}{{\mathcal{K}}}

\newcommand{\ba}{\begin{align}}
\newcommand{\ea}{\end{align}}
\newcommand{\bea}{\begin{eqnarray}}
\newcommand{\eea}{\end{eqnarray}}
\newcommand{\be}{\begin{equation}}
\newcommand{\ee}{\end{equation}}

\newcommand{\eg}{{\it e.g.~}}

\newcommand{\Js}{\mathcal{J}}

\theoremstyle{plain}
\newtheorem{theorem}{Theorem} 

\newtheorem{lemma}[theorem]{Lemma}
\crefname{lemma}{Lemma}{Lemmas}
\newtheorem{proposition}[theorem]{Proposition}
\crefname{proposition}{Proposition}{Propositions}
\newtheorem{corollary}[theorem]{Corollary}
\crefname{corollary}{Corollary}{Corollaries}

\theoremstyle{plain}
\newtheorem{conj}{Conjecture}
\theoremstyle{plain}

\theoremstyle{plain}

\newtheoremstyle{ColonNotDot}%
{}
{}
{\itshape}
{}
{\bfseries}
{:}
{ }
{}

\theoremstyle{plain}
\newtheorem{definition}{Definition}
\theoremstyle{definition}

\theoremstyle{definition}

\newcommand\summaryname{Abstract}
{\small\begin{center}%
		\bfseries{\summaryname} \end{center}}

\title{On the Space of Slow Growing Weak Jacobi Forms}
\author{ Christoph A.~Keller}
\address{Christoph A. Keller,  Department of Mathematics, University of Arizona, Tucson, AZ 85721-0089, USA}
\email{cakeller@math.arizona.edu}

\author{Jason M.~ Quinones}
\address{Jason M. Quinones, School of STAMP, Gallaudet University, Washington DC 20002, USA}
\email{jason.quinones@gallaudet.edu}

\begin{document}

\begin{abstract}
    Weak Jacobi forms of weight 0 and index $m$ can be exponentially lifted to meromorphic Siegel paramodular forms. It was recently observed that the Fourier coefficients of such lifts are then either fast growing or slow growing. In this note we investigate the space of weak Jacobi forms that lead to slow growth. We provide analytic and numerical evidence for the conjecture that there are such slow growing forms for any index $m$.
\end{abstract}

\maketitle


\section{Introduction}
In this note we investigate the asymptotic behavior of the Fourier coefficients of certain types of automorphic forms, namely weak Jacobi forms \cite{MR781735} and their exponential lifts, which give meromorphic Siegel (para-)modular forms \cite{MR1616929}.
The growth of coefficients of automorphic forms has of course long been of interest in mathematics: the most famous example may be the work of Hardy and Ramanujan, who found asymptotic expressions for the number of integer partitions from the Dedekind $\eta$-function \cite{MR2280879}. However, more recently, the growth of coefficients of automorphic forms has also attracted interest in physics: it describes the entropy of certain types black holes \cite{Strominger:1996sh,Sen2011a}. This was explored from a more mathematical perspective for instance in \cite{Dabholkar:2012nd}. The main motivation for this works comes from a series of articles \cite{Belin:2019jqz, Belin:2019rba,Belin:2020nmp} that were written in this context. There it was discovered that weak Jacobi forms could either lead to fast growth or slow growth in their exponential lifts. Our goal is to investigate the space of slow growing forms.

Let us introduce the main definition of this article.
Let $\varphi$ be a  weak Jacobi form (wJf) \cite{MR781735} of weight 0 and index $m$ with Fourier expansion
\be
\varphi(\tau,z) =\sum_{n\in\Z_{\geq0},l\in \Z} c(n,l) q^n y^l\ .
\ee
Given a term $q^n y^l$ in this expansion, we call the quantity $l^2-4mn$ its \emph{polarity}, and we say that the term is \emph{polar} if its polarity is positive.
Let $q^a y^b$ be such a polar term of $\varphi$. We want to study the functions
\begin{align} \label{eq:sumsofcoeffs}
f_{a,b}(n,l):=\underset{r \in \mathbb{Z}}{\sum} c(nr+ar^2,l-br)
\end{align}
and their asymptotic growth as $n$ becomes large. Note that this sum is a finite sum because $c(n,l)=0$ if $l^2-4mn>m^2$. We will explain the motivation behind this definition below.
We make the following definition: 
\begin{definition}
Let $m\in\N$ and $a,b\in\N_0$ with $b^2-4ma >0$. We say a weak Jacobi form $\varphi$ of index $m$ and weight 0 is \emph{slow growing about $q^ay^b$} if it has no terms of polarity greater than $b^2-4ma$ and $f_{a,b}(n,l)$ is bounded as a function of $n$ and $l$. We denote the vector space of wJf that are slow growing about $q^ay^b$ by $\Js^{a,b}_m$.
\end{definition}
The purpose of this article is to describe and explore the spaces $\Js^{a,b}_m$. Our ultimate goal is to prove the following conjecture:
\begin{conj}\label{conjMain}
For every index $m$, there is at least one choice of $a,b$ such that $\dim \Js^{a,b}_m > 0$.
\end{conj}
To explain our interest in this conjecture, note that it is somewhat surprising that slow growing forms exist in the first place: If a wJf has a term of maximal polarity $q^ay^b$, then its coefficients grow like
\be
|c(n,l)| \sim \exp \pi \sqrt{ \frac{b^2-4ma}{m^2}(4mn-l^2)}\
\ee
in the limit of large discriminant $4mn-l^2 \gg 0$. 
For (\ref{eq:sumsofcoeffs}) to be bounded, the exponentially large terms thus have to cancel out almost completely, which is not something we would expect for a generic form $\varphi$. However, there are indeed examples where such cancellations happen. The simplest is $\varphi_{0,1}$, the unique (up to normalization) wJf of weight 0 and index 1. In this note we provide evidence that such slow growing forms exist for every index $m$. Even though we do not manage to prove this conjecture, we gather numerical and analytical evidence in favor of it.

Let us now explain where definition~(\ref{eq:sumsofcoeffs}) comes from, and why we are interested in slow growing wJf in the first place. 
A wJf of weight 0 can be exponentially lifted to a meromorphic Siegel paramodular form \cite{MR1616929}. It is the Fourier coefficients of this Siegel form that we want to study.
Since this form meromorphic,  we need to be careful to specify which region we are expanding in when defining its Fourier coefficients. Its poles are given by Humbert surfaces specified by $a,b$. $f_{a,b}(n,l)$ then describes the growth of the Fourier coefficients in an appropriate limit when expanded around that particular Humbert surface \cite{Sen2011a,Belin:2019jqz}.
The most basic example of such a lift is the wJf $\varphi_{0,1}$, which is lifted to the reciprocal of the Igusa cusp form $\chi_{10}$. This is the case originally studied by physicists to describe the entropy of certain black holes \cite{Strominger:1996sh}, and was further explored in \cite{Dabholkar:2012nd}. Our results are a generalization of this case.

We note that for the purpose of obtaining Fourier coefficients of SMF, strictly speaking  we are not interested in $\Js^{a,b}_m$, but rather in the subset
\be
\hat\Js^{a,b}_m := \{ \varphi \in \Js^{a,b}_m : c(a,b)=1 \}\ .
\ee
A form in $\hat\Js^{a,b}_m$ is guaranteed to lead to a (simple) pole in the corresponding exponentially lifted Siegel form.
$\hat\Js$ is an affine space. If it is non-empty, that is if $\Js$ contains at least one form with $c(a,b)=1$, then  $\dim \hat\Js = \dim \Js -1$, since the linear constraint $c(a,b)=0$ on $\Js$ has rank 1. It turns out this is almost always what happens. In the following we will therefore always give $\dim \Js$, and indicate explicitly the few cases when $\hat\Js = \emptyset$.

The Fourier coefficients of Siegel modular forms give one way in which the $f_{a,b}(n,l)$ are connected to automorphic forms. Let us mention that for $a=0$, there is also another connection: in \cite{Belin:2019jqz} it was shown that the $f_{0,b}(n,l)$ are Fourier coefficients of certain vector valued modular functions. In fact this gave an explicit test for when a wJf is slow growing, and very simple expressions for the $f_{0,b}(n,l)$ in this case --- see theorem~\ref{prop:GeneratingFunctions} below. There is no analogue test for $a>0$, and the situation is thus much less clear. We discuss the case $a>0$ in section~\ref{s:qayb} and find some evidence that the $f_{a,b}$ could again be coefficients of some modular object.

Finally let us point out that slow growth seems to be related to the maximal polarity of the terms appearing in a form.
In \cite{Belin:2019jqz} for instance it was established that $\varphi$ can only be slow growing about $q^0 y^b$ if $b^2 \leq m$. In section~\ref{s:qayb} we give examples of forms growing slowly about $q^ay^b$ whose maximal polarity is $b^2-4ma > m$, but only slightly so. 
The implication that slow growing forms can only have terms of relatively low polarity. It thus makes sense to look for forms of low polarity. A priori it is not clear that such wJf even exist for any $m$, regardless of any question of slow growth: The issue is that even though for weight 0 forms, the polar terms uniquely determine $\varphi$, it is not guaranteed that for a choice of polar terms there exists a corresponding wJf $\varphi$. It is believed \cite{Gaberdiel:2008xb} that for any $m$, there exist wJfs whose most polar term has polarity around $m/2$. In particular this allows for the existence of slow growing wJf.
In section~\ref{s:polar} we provide strong numerical evidence for this belief, and in particular establish the existence of such forms up to $m=1000$. We note that a similar question is discussed in \cite{Dabholkar:2012nd}, where so-called optimal wJf $\cK_m\in J_{2,m}$ are constructed, whose maximal polarity is 1. It is natural to believe that $\varphi_{-2,1}\cK_m$ is then a slow growing form for index $m+1$, which is indeed true for the cases that we checked. A proof of this belief would of course prove conjecture~\ref{conjMain}.

In section~\ref{s:yb} we then discuss weak Jacobi forms that have slow growth around $q^0y^b$. Following up on an observation in \cite{Belin:2020nmp}, we give infinite families of weak Jacobi forms based on quotients of theta functions that are all slow growing. We obtain bounds on the $\dim \Js^{0,b}_m$, and compute these dimensions explicitly up to $m=61$. Altogether we establish that slow growing forms exist for every index up to $m=78$.

In section~\ref{s:qayb} we give some analytical and numerical results on the growth about terms $q^a y^b$. This case is harder because unlike $a=0$, there is no longer a straightforward criterion to determine if a form is slow growing. Nonetheless, we manage to prove slow growth for certain cases. In the process we clarify a question raised in \cite{Belin:2018oza}, where an exponentially lifted Siegel modular form with two Humbert surface of the same discriminant was considered. The physical expectation was that the expansion around both surfaces should essentially give the same coefficients, which is what we prove here.
For many other cases we give strong numerical evidence for or against slow growth by evaluating the first few coefficients, giving a better idea on the form of $\Js^{a,b}_m$. We give examples of forms growing slowly about $q^ay^b$ whose maximal polarity is $b^2-4ma > m$, and we also give examples that are slow growing about one term but fast growing about another term of the same polarity, thus establishing that the polarity of the term is not enough to determine the type of growth.

\emph{Acknowledgments}: We thank Nathan Benjamin for useful discussions.
The work of C.A.K. is supported in part by the Simons Foundation Grant No.~629215.

\section{Polar Terms of Weak Jacobi Forms}\label{s:polar}
Let $J_{0,m}$ be the space of weak Jacobi forms of weight $0$ and index $m$ \cite{MR781735}. Its structure is very simple: the ring of wJf of weight 0 is freely generated by three forms $\phi_{0,1}$, $\phi_{0,2}$, $\phi_{0,3}$ given \eg in \cite{Gritsenko:1999fk}. $J_{0,m}$ is then spanned by products $\phi_{0,1}^a\phi_{0,2}^b\phi_{0,3}^c$, with $a+2b+3c=m$.

In practice however, working with wJf of already moderate index can be quite cumbersome, since it involves expanding products of series expansions to fairly high order. It is therefore important to have computationally efficient expressions for the generators. We used
\begin{align}\label{phi01}
\phi_{0,1}&=4\left(\frac{\theta_{2}(\tau,z)}{\theta_{2}(\tau,0)}^2+\frac{\theta_{3}(\tau,z)}{\theta_{3}(\tau,0)}^2+\frac{\theta_{4}(\tau,z)}{\theta_{4}(\tau,0)}^2\right), \\
\label{phi02}
\phi_{0,2}&=\frac{1}{2}\eta(\tau)^{-4} \underset{m,n \in \mathbb{Z}}{\sum} (3m-n) (\frac{-4}{m})(\frac{12}{n})q^{\frac{3m^2+n^2}{24}}y^{\frac{m+n}{2}}, 
\end{align}
and 
\begin{align}\label{phi03}
    \phi_{0,3}=\left(\frac{q^{1/24}}{\eta(\tau)}\left(\underset{l \in\mathbb{Z}}{\sum}q^{6l^2+l}y^{\frac{12l+1}{2}}+\underset{l\in\mathbb{Z}}{\sum}q^{6l^2-l}y^{\frac{12l-1}{2}}-\underset{l\in\mathbb{Z}}{\sum}q^{6l^2+5l+1}y^{\frac{12l+5}{2}}-\underset{l\in\mathbb{Z}}{\sum}q^{6l^2-5l+1}y^{\frac{12l-5}{2}}\right)\right)^2.
\end{align}
We are using this novel form because it can be evaluated more quickly, as it requires only few multiplications, and no divisions by expressions in multiple variables.

As mentioned in the introduction, wJf cannot have slow growth about $q^0 y^b$ if $b>\sqrt{m}$. We therefore want to investigate wJf with no terms of high polarity. 
To this end we define
\begin{definition}\label{defPm}
	Let $J^P_{0,m}:= \{ \varphi_{0,m} \in J_{0,m} \mid c(n,l)=0 \text{ for } l^2-4mn > P \}$. Define \textit{$P(m)$} to be the integer $P$ such that $J_{0,m}^{P}=0$ and $J_{0,m}^{P+1} \neq 0$.
\end{definition}
$P(m)$ is well-defined since $J^P_{0,m} \subset J^{P+1}_{0,m}$, $J^0_{0,m} =0$ for $m>0$, and moreover $J^{m^2}_{0,m}=J_{0,m}$.
In particular definition~\ref{defPm} implies that there is a non-zero wJf whose most  polar term has polarity $l^2-4mn = P(m)$, but no no-zero wJf that only has terms of polarity strictly smaller than $P(m)$. 

Given the generators of $J_{0,m}$, in principle it is straightforward to compute $P(m)$. In practice this involves expanding power series to high order to read off polar terms, which becomes computationally expensive very quickly. Using (\ref{phi01}-\ref{phi03}),
we computed $P(m)$ to index $m=61$. The result is plotted in figure~\ref{figure:ScatterPlotP(m)}.

\begin{figure}[H]
	\centering
	\includegraphics[]{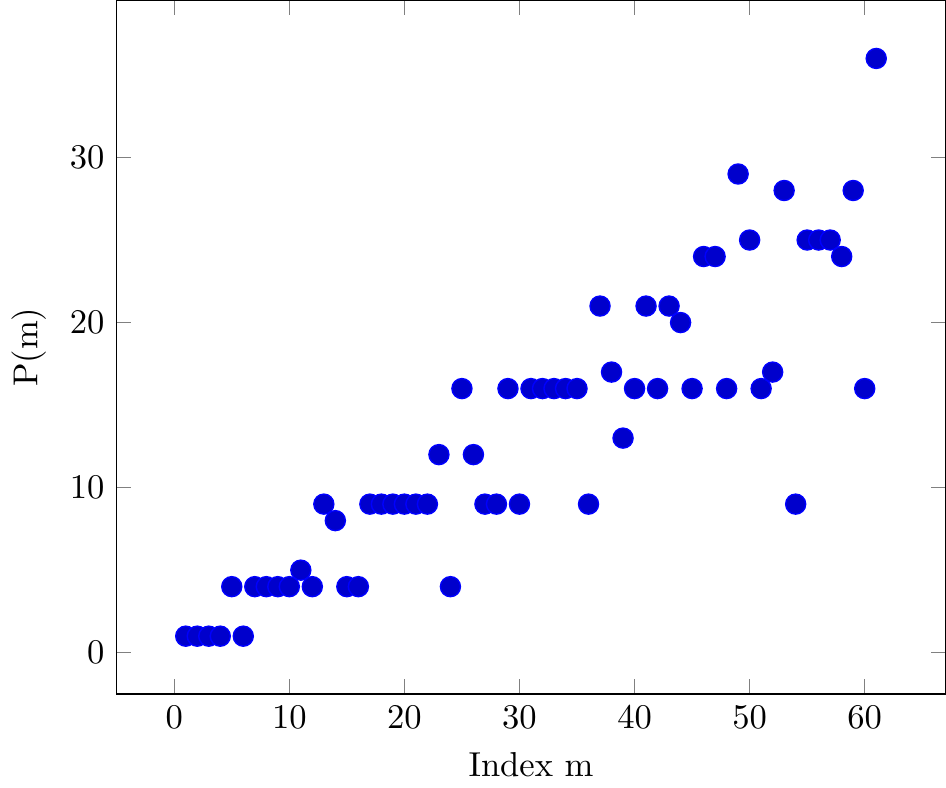}
	\caption{Scatterplot of $P(m)$, the smallest polarity such that there exists a weak Jacobi form with $P(m)$ as its most polar term.}
	\label{figure:ScatterPlotP(m)}
\end{figure}
We are interested in the asymptotic behavior of $P(m)$, which unfortunately is hard to compute directly.
For large $m$, \cite{Gaberdiel:2008xb} conjectured the following form for $P(m)$:
\begin{conj}\label{conjPm}
As $m\to\infty$, $P(m) = \frac m2 + O(m^{1/2})$.
\end{conj}

In what follows, we will give upper and lower bounds $P_+(m)\geq P(m) \geq P_-(m)$ for $P(m)$. The bounds we find are indeed compatible with conjecture~\ref{conjPm}.
Let us first discuss the lower bound $P_-(m)$ for $P(m)$. To this end, we prove the following proposition, which strengthens a result in \cite{Manschot:2008zb}:
\begin{proposition} \label{prop:PolarityDeterminesWjF}
 For $m>0$,	the polar terms $c(n,l)$ of polarity $l^2-4mn \geq m/6$ uniquely determine the weak Jacobi form $\varphi_{0,m}$.
\end{proposition}
\begin{proof} 
	Consider the theta decomposition \cite[equation (5.5)]{MR781735} of the weak Jacobi form, \begin{align} \label{eq:ThetaDecomp} \varphi_{0,m}(\tau,z)=\underset{\mu \text{ mod } 2m \in \mathbb{Z}/2m \mathbb{Z}}{\sum} h_\mu(\tau) \theta_{m,\mu}(\tau,z). \end{align} The polar terms of $\varphi_{0,m}$ appear in the negative $q$-power part of the Fourier expansions of $\{ h_\mu(\tau) : \mu \in \mathbb{Z}/2m \mathbb{Z} \}$. 
	
	First we show that the product $\eta(\tau)h_\mu(\tau)$, where $\eta$ is the Dedekind eta function, is a scalar modular form of weight zero for the congruence group $\Gamma\big(\text{lcm}(24,4m)\big)$. 
    The Dedekind eta function $\eta(\tau)=q^{1/24} \underset{m >0}\prod (1- q^m)$ is a scalar modular form of weight $1/2$ for $\Gamma(24)$ with multiplier system $(\frac{c}{d})$. The forms $h_\mu(\tau)$ are scalar modular forms of weight $-1/2$ for $\Gamma(4m)$ with the same multiplier system $(\frac{c}{d})$.
	Since $(\frac{c}{d})$ squares to the identity, the product $\eta(\tau) h_\mu(\tau)$ is a scalar modular form for $\Gamma\big(\text{lcm}(24,4m)\big)$ with trivial multiplier system.
	
	Now, given a weak Jacobi form with no polar terms of polarity greater than or equal to $m/6$, we show this form must be identically zero. Let $N$ be the maximal polarity of this weak Jacobi form, this value shall also be the maximal polarity of $h_\mu(\tau)$ in its theta decomposition. The Fourier expansion of $\eta(\tau)h_\mu(\tau)$ then begins at $c(N,\mu)q^{-N/4m + \frac{1}{24}}$. We have $N < \frac{m}{6}$ by assumption, so
	\begin{align}\label{polcond}
	-N/4m + \frac{1}{24} > 0,
	\end{align} which implies that $\eta(\tau)h_\mu(\tau)$ has no singularity at $q=0$. It is thus indeed an element of $M_0\left(\Gamma\big(\text{lcm}(24,4m)\big)\right)$. However, the only modular forms in $M_0\left(\Gamma\big(\text{lcm}(24,4m)\big)\right)$ are constants, whose Fourier expansion consists of only the $q^0$ term. Since (\ref{polcond}) rules out a constant term, $\eta(\tau)h_\mu(\tau)$ is zero.
\end{proof}
From this it follows immediately that
\begin{corollary}
	 $P_-(m):=\lceil \frac{m}{6} \rceil$ is a lower bound for $P(m)$. 
\end{corollary}

Next let us discuss upper bounds $P_+(m)$. One way to obtain such a bound is the following result of \cite{Dabholkar:2012nd}:
\begin{theorem} \cite[Theorem 9.4]{Dabholkar:2012nd}
For each $m \in \N$, there exists a wJf $\cK_{m-1}\in J_{2,m-1}$ whose maximal polarity is 1. 
\end{theorem}
From this we obtain the following corollary: 
\begin{corollary}
The wJf $\varphi_{-2,1}\cK_{m-1}\in J_{0,m}$ has no terms of polarity greater than $m+1$. Here $\varphi_{-2,1}$ is the unique wJf of weight -2 and index 1 whose maximal polar term is $y$ as defined in \cite[p.108]{MR781735}. 
\end{corollary}
\begin{proof}

The polarity of a term in the product of two wJf of different indexes is not so straightforward, so we must prove $\varphi_{-2,1}\cK_{m-1}\in J_{0,m}$ cannot have polarity exceeding $m+1$. 
To this end, let $q^Ny^L$ be a term of $\varphi_{-2,1}$ whose polarity is $\Delta$, and let $q^n y^l$ be a term of $\cK_{m-1}$ whose polarity is $\Delta'$. Then the term $q^{N+n}y^{L+l}$ of $\varphi_{-2,1}\cK_{m-1}$ has polarity
\begin{align}
m\Delta'+(1+\frac{1}{m-1})\Delta-\frac{1}{m-1}\left((m-1)L-l\right)^2 \leq m+1-\frac{1}{m-1}\left((m-1)L-l\right)^2,
\end{align}
where we used the bounds $\Delta, \Delta' \leq 1$. The maximum value of the right hand side is $m+1$, occurring when $(m-1)L=l$. 
\end{proof} 
It follows that $P_+(m)=m+1$ is an upper bound for $P(m)$. Unfortunately, this bound is slightly too weak for our purposes, since we would like to find wJf of polarity $m$ or less. Nonetheless, we note that the forms obtain in this way are natural candidates for being slow growing. In fact, for the first few low lying values of $m$ we checked explicitly that they are slow growing. This is the case even for those with maximal polarity equal to $m+1$, in which case the term of maximal polarity is $q^ay^b$ with $a>0$.

More importantly, the bound $P_+(m)=m+1$ also seems to be quite far from optimal when compared to conjecture~\ref{conjPm}. Let us therefore describe an alternative approach, which will lead to a tighter numerical bound.
For this we use the following counting argument:
\begin{lemma} Let $p_{\mathscr{P}}(m):=\overset{m}{\underset{l = \lceil \sqrt{\mathscr{P}} \rceil}{\sum}} \lceil \frac{l^2 - \mathscr{P}}{4m} \rceil$ be the number of polar terms with polarity greater or equal to $\mathscr{P}$. For any $\mathscr{P}$ satisfying the inequality $p_{\mathscr{P}}(m) \leq j(m) < p_{\mathscr{P} +1}(m)$, we have
\begin{align}
P(m) \leq \mathscr{P} \ ,
\end{align}
where $j(m) = \dim J_{0,m}$.  Denote by $P^+(m)$ the smallest such $\mathscr{P}$.
\end{lemma}
\begin{proof} The polar terms for a given index $m$ form a linear system for the space of weak Jacobi forms of index $m$. We order the polar terms according to their polarity. We may use the $j(m)$ basis elements to set $(j(m)-1)$ of the most polar terms to zero, so that $P(m)$ is bounded above by the value of $\mathscr{P}$ such that $p_{\mathscr{P}}(m) \leq j(m) < p_{\mathscr{P} +1}(m)$.
\end{proof}

Since $P^+(m)$ only involves counting polar terms, it is very easy to compute. We present a scatter plot of its value for $1 \leq m \leq 1000$ in Figure \ref{figure:ScatterPlotUpperBound}. 

We expect that for generic $m$, $P(m)$ will be very close to $P^+(m)$. Equality between $P(m)$ and $P^+(m)$ holds whenever the linear system of polar coefficients with polarity greater than or equal to $P^+(m)$ has maximal rank. We expect the matrix of these polar coefficients to behave like a random matrix, and such matrices generically have maximal rank. Comparing Figure \ref{figure:ScatterPlotUpperBound} with the scatter plot for $P(m)$ for $1 \leq m \leq 61$ in Figure \ref{figure:ScatterPlotP(m)}, we find that $P(m)=P^+(m)$ except at $m=39, 51, 54$, and $58$. A particularly wide gap is found at $m=54$, where $P(m)=9$ but $P^+(m)=25$.
In contrast to $P^+(m)$, we expect $P^-(m)$ to be a weak lower bound.
Numerically, up to $m=1000$ we find that $|P^+(m)-\frac{m}{2}| \leq 2.1016 m^{1/2}$.
\begin{figure}[H]
	\centering
	\includegraphics[]{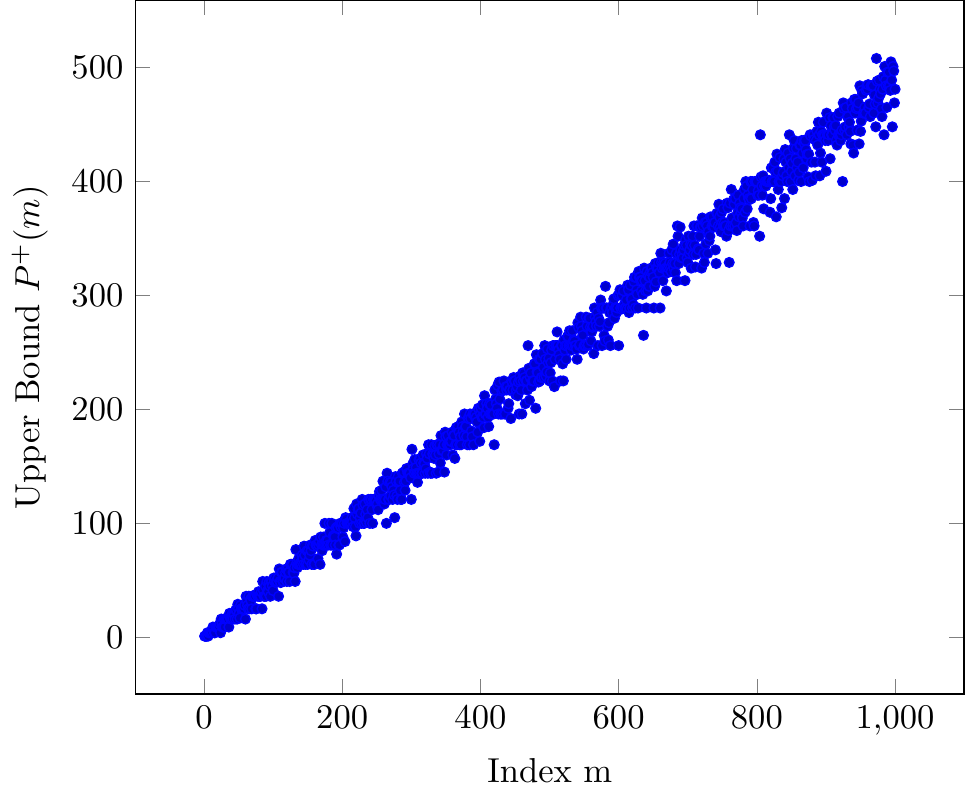}
	\caption{Scatterplot of the upper bound $P^+(m)$ for $P(m)$, where $P^+(m)$ is the polarity such that the number of polar terms of polarity $\geq P^+(m)$ equals $j(m)$.}
	\label{figure:ScatterPlotUpperBound}
\end{figure}

\section{Slow Growth about $y^b$ and Theta Quotients}\label{s:yb}
\subsection{Dimension}
Let us now explore the space of slow growing wJf $\Js^{a,b}_m$ in more detail.
In this section we will focus on slow growing wJf around $q^0 y^b$. For this case, there exists an explicit criterion to determine if a given wJf is slow growing or not:
\begin{proposition} \label{prop:GeneratingFunctions}
	\cite[(4.16)]{Belin:2019jqz} 
	Introducing variables $n_b = n\mod b$, $M=mn^2/b^2+nl/b$, $k= 2(n-n_b)m/b+l$, and writing $f_{0,b}(n,l)=f_{0,b}(M,n_b,k)$, the generating functions 
	 \be
 F_{n_b,k}(\tau) = \sum_M f_{0,b}(M,n_b,k)q^M\ .
 \ee
	 are given by 
	\begin{align}
	F_{n_b,k}(\tau)=\frac{1}{b} \overset{b-1}{\underset{j=0}{\sum}} \chi_{n_b,j}(\tau) e^{-2 \pi i k j/b}  .
	\end{align}
 Here, the $\chi_{n_b,j}(\tau)=q^{mn_b^2/b^2} \varphi(\tau,(n_b \tau+j)/b)$,  $n_b=0,\dots,b-1$ and $j=0,\dots,b-1$, are specializations of $\varphi$. The wJf $\varphi$ is slow growing around $q^0y^b$ iff all $\chi_{n_b,j}(\tau)$ are regular at $q=0$ . 
\end{proposition}
The proof of this is based on the fact that by \cite[Theorem 1.3]{MR781735}, the $\chi_{n_b,j}(\tau)$ are modular forms of some congruence subgroup. It follows that if all the $\chi_{n_b,j}(\tau)$ are regular at $q=0$, they are all constant, which implies that $\varphi$ is slow growing. For such slow growing forms, one can give explicit expressions for the $f_{0,b}(n,l)$: they vanish unless $mn+bl=0$ or $n=0$, and only take a finite number of different values otherwise, thus being manifestly bounded. In practice, the following corollary can provide a more useful test for slow growth (`$\alpha$-test'):

\begin{corollary} \label{cor:AlphaTest} \cite[(5.2)]{Belin:2019jqz}
Irregularity of $\chi_{n_b,j}(\tau)$ at $q=0$ comes from the presence of polar terms $q^n y^l$ in the wJf $\varphi$ for which the value 
\begin{align}
    \alpha:=\underset{j=0,\dots,b-1}{\max} \left[-m(\frac{j}{b}-\frac{l}{2m})^2-\frac{1}{4m}(4mn-l^2)\right]
\end{align}
is such that $\alpha >0$. Thus, the wJf $\varphi$ is slow growing about $q^0y^b$ iff $\varphi$ has no polar term $q^ny^l$ with $\alpha>0$. 
\end{corollary}

We can apply this $\alpha$-test to compute the dimension of the space of slow growing wJf $\dim \Js^{0,b}_m$. Checking regularity is straightforward, but quickly becomes computationally expensive for the reason explained in section~\ref{s:polar}. 
Using the computationally efficient basis (\ref{phi01}-\ref{phi03}), we were able compute this dimension for all values of $b$ up to $m=61$.
We list the results in table~\ref{table:ybTable} and plot it in figure~\ref{f:slowGrowth}. Note that the $\alpha$-test shows that any term $y^b$ with $b>\sqrt{m}$ leads to fast growth, so that we only need to test up to that value. 
Further, we find that for all $m,b$ listed with $\dim \Js^{0,b}_m \neq 0$, except for the case $m=54,b=4$, the corresponding affine space $\hat\Js^{0,b}_m$ is nonempty. The sole exception, $m=54,b=4$, has $\Js^{0,4}_{54}=1$ and  $\hat\Js^{0,4}_{54}=\emptyset$.
In view of conjecture~\ref{conjMain}, we note that for every $m$ in the table there is a $b$ with a non-vanishing slow growing wJf. 

\begin{figure}[H]
	\centering
\includegraphics[]{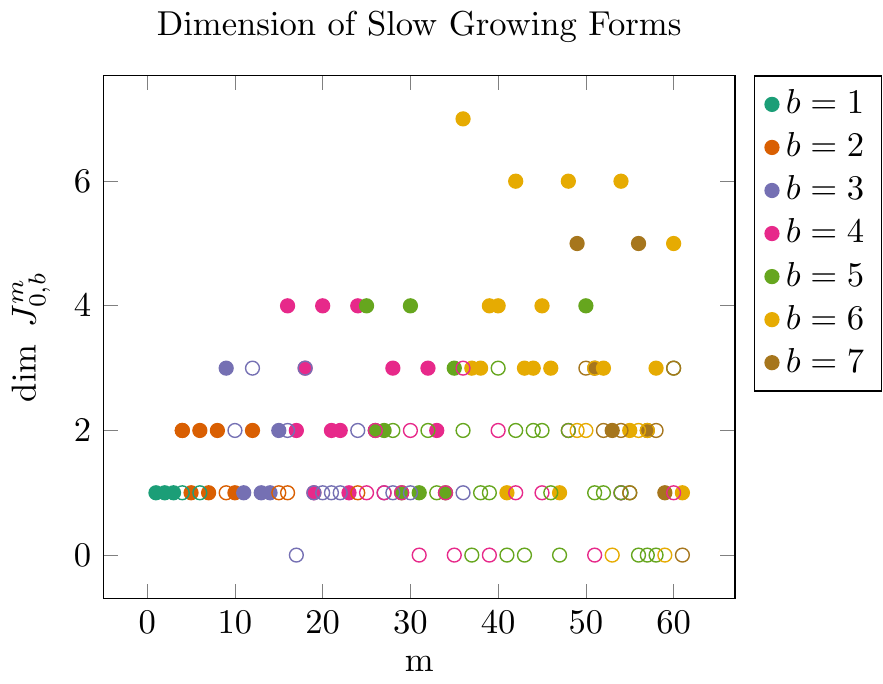}
\caption{$\dim \Js^m_{0,b}$. For each index $m$, we denote different values of $b$ by empty circles of different color. The solid circle corresponds to  $\max_b \dim \Js^m_{0,b}$.} 
\label{f:slowGrowth}
\end{figure}

\begin{table}
\centering
\begin{tabular}{ccc |c cc |c cc |ccc |ccc}
m  & b & dim $\Js^m_{0,b}$ & m  & b & dim $\Js^m_{0,b}$ & m  & b & dim $\Js^m_{0,b}$ & m  & b & dim $\Js^m_{0,b}$& m  & b & dim $\Js^m_{0,b}$ \\
1  & 1 & 1   & 19 & 3 & 1   & 31 & 5 & 1   & 42 & 6 & 6   & 53 & 6 & 0   \\
2  & 1 & 1   & 19 & 4 & 1   & 32 & 4 & 3   & 43 & 5 & 0   & 53 & 7 & 2   \\
3  & 1 & 1   & 20 & 3 & 1   & 32 & 5 & 2   & 43 & 6 & 3   & 54 & 3 & 1   \\
4  & 1 & 1   & 20 & 4 & 4   & 33 & 4 & 2   & 44 & 5 & 2   & 54 & 4 & 1   \\
4  & 2 & 2   & 21 & 3 & 1   & 33 & 5 & 1   & 44 & 6 & 3   & 54 & 5 & 1   \\
5  & 2 & 1   & 21 & 4 & 2   & 34 & 4 & 1   & 45 & 4 & 1   & 54 & 6 & 6   \\
6  & 1 & 1   & 22 & 3 & 1   & 34 & 5 & 1   & 45 & 5 & 2   & 54 & 7 & 2   \\
6  & 2 & 2   & 22 & 4 & 2   & 35 & 4 & 0   & 45 & 6 & 4   & 55 & 5 & 1   \\
7  & 2 & 1   & 23 & 4 & 1   & 35 & 5 & 3   & 46 & 5 & 1   & 55 & 6 & 2   \\
8  & 2 & 2   & 24 & 2 & 1   & 36 & 3 & 1   & 46 & 6 & 3   & 55 & 7 & 1   \\
9  & 2 & 1   & 24 & 3 & 2   & 36 & 4 & 3   & 47 & 5 & 0   & 56 & 5 & 0   \\
9  & 3 & 3   & 24 & 4 & 4   & 36 & 5 & 2   & 47 & 6 & 1   & 56 & 6 & 2   \\
10 & 2 & 1   & 25 & 4 & 1   & 36 & 6 & 7   & 48 & 4 & 2   & 56 & 7 & 5   \\
10 & 3 & 2   & 25 & 5 & 4   & 37 & 5 & 0   & 48 & 5 & 2   & 57 & 5 & 0   \\
11 & 3 & 1   & 26 & 4 & 2   & 37 & 6 & 3   & 48 & 6 & 6   & 57 & 6 & 2   \\
12 & 2 & 2   & 26 & 5 & 2   & 38 & 5 & 1   & 49 & 6 & 2   & 57 & 7 & 2   \\
12 & 3 & 3   & 27 & 3 & 1   & 38 & 6 & 3   & 49 & 7 & 5   & 58 & 5 & 0   \\
13 & 3 & 1   & 27 & 4 & 1   & 39 & 4 & 0   & 50 & 5 & 4   & 58 & 6 & 3   \\
14 & 3 & 1   & 27 & 5 & 2   & 39 & 5 & 1   & 50 & 6 & 2   & 58 & 7 & 2   \\
15 & 2 & 1   & 28 & 3 & 1   & 39 & 6 & 4   & 50 & 7 & 3   & 59 & 6 & 0   \\
15 & 3 & 2   & 28 & 4 & 3   & 40 & 4 & 2   & 51 & 4 & 0   & 59 & 7 & 1   \\
16 & 2 & 1   & 28 & 5 & 2   & 40 & 5 & 3   & 51 & 5 & 1   & 60 & 4 & 1   \\
16 & 2 & 2   & 29 & 4 & 1   & 40 & 6 & 4   & 51 & 6 & 3   & 60 & 5 & 3   \\
16 & 4 & 4   & 29 & 5 & 1   & 41 & 5 & 0   & 51 & 7 & 3   & 60 & 6 & 5   \\
17 & 3 & 0   & 30 & 3 & 1   & 41 & 6 & 1   & 52 & 5 & 1   & 60 & 7 & 3   \\
17 & 4 & 2   & 30 & 4 & 2   & 42 & 4 & 1   & 52 & 6 & 3   & 61 & 6 & 1   \\
18 & 3 & 3   & 30 & 5 & 4   & 42 & 5 & 2   & 52 & 7 & 2   & 61 & 7 & 0   \\
18 & 4 & 3   & 31 & 4 & 0   &    &   &     &    &   &     &    &   &    
\end{tabular}
\caption{Dimension of the Space of Weak Jacobi Forms of Weight 0 and Index $m$ that are Slow Growing About Their Most Polar $y^b$ Term}
\label{table:ybTable}
\end{table}

\subsection{Lower bound}
Since for larger $m$ constructing the explicit space of slow growing wJf is computationally expensive, we will instead use the same strategy as in section~\ref{s:polar}: we will give a lower bound on the dimension by computing the size of a constraint matrix.

\begin{proposition} For index $m$ and integer $b$, let $\rho(m,b)$ be the number of polar terms $q^n y^l$ with either $\alpha >0$ or $4mn-l^2<-b^2$. The dimension of $\Js^{0,b}_m$ is bounded below by $j_-^{0,b}(m):=j(m) - \rho(m,b)$. \end{proposition}
\begin{proof}
	$\Js^{0,b}_m$ is the space of wJfs that have no polar terms $q^ny^l$ with either $\alpha>0$ or discriminant $4mn-l^2<-b^2$. This can be encoded as a linear system, with respect to the polar coefficients of a basis of $J_{0,m}$. Indeed, given a basis of $J_{0,m}$, let $A$ be the matrix where the $j$-th row is the polar coefficients $c(n,l)$ with $4mn-l^2<-b^2$ and the polar coefficients of terms $q^ny^l$ with $\alpha >0$ of the $j$-th basis element. The linear system is
	\begin{align}
	Ax=0
	\end{align}
	
	The space of solutions to the linear system above is $\Js^{0,b}_m$. By rank-nullity, $\text{null}(A) \geq j(m)-\rho(m,b)$.  

\end{proof}

Since this bound only relies on counting the number of terms, we can easily evaluate it for higher values of $m$. In figure~\ref{f:boundSlow}, for all $m$ up to 500 we plot 
\be
j_-(m) := \max_b j_-^{0,b}(m)\ ,
\ee
the maximum over all $b$ of the dimension $\Js^{0,b}_m$. If this lower bound were optimal, then figure~\ref{f:boundSlow} would be problematic for conjecture~\ref{conjMain}: for instance, it would predict no slow growing forms about any $y^b$ for $m=41,47,59$. However, for all such cases we checked, the bound is actually not optimal, as can be seen by comparing to table~\ref{table:ybTable}. Our conclusion is that the bound probably becomes less and less optimal as $m$ grows.

\begin{figure}
	\centering
\includegraphics[]{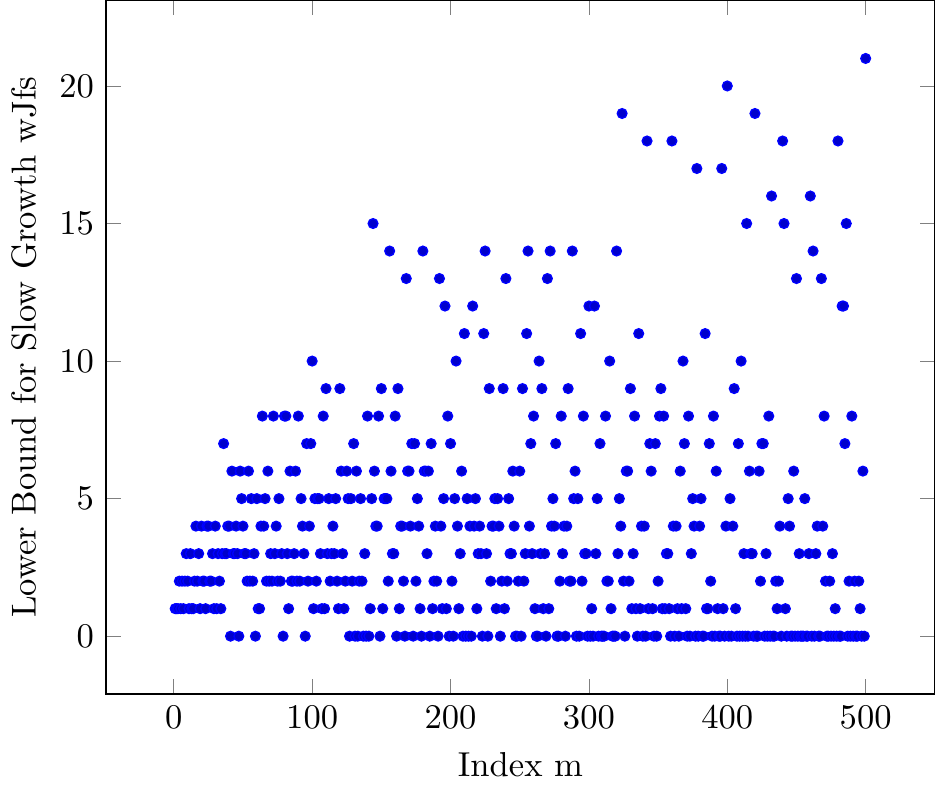}
\caption{Scatterplot of the lower bound $j_-(m)$ for the dimension of weak Jacobi forms that are slow growth about their most polar term $y^b$, for any $b$.
\label{f:boundSlow}}
\end{figure}

\subsection{Theta quotients}
Since the lower bound $j_-(m)$ is probably not optimal for large values of $m$, we will also pursue a different strategy, namely finding explicit families of slow growing wJf.
For this we introduce the theta function 
\begin{align} \label{eq:GritsenkoTheta}
\theta_1(\tau,z)=-q^{1/8}y^{-1/2}\overset{\infty}{\underset{n=1}{\prod}}(1-q^n)(1-q^{n-1}y)(1-q^ny^{-1}).
\end{align}
The theta function $\theta_1(\tau,z)$ is a  weak Jacobi form of weight $\frac{1}{2}$ and index $\frac{1}{2}$ with multiplier system $\nu_\eta^3 \times \nu_H$, where $\nu_\eta$ is the multiplier system of the $\eta$ function (see for example \cite[p.15]{MR2766155}).  To obtain a weak Jacobi form of integer index and trivial multiplier system, we will take quotients of $\theta_1(\tau,\alpha z)$ where $\alpha$ is some scaling factor. More precisely, we consider
\begin{align*}
\overset{N}{\underset{j=1}{\prod}} \frac{\theta_1(\tau, n_1 z)\theta_1(\tau,n_2 z) \cdots \theta_1(\tau,n_N z)}{\theta_1(\tau, m_1 z)\theta_1(\tau,m_2 z) \cdots \theta_1(\tau,m_N z)}.
\end{align*}
Since there is an equal number of thetas in the numerator and denominator, the quotient has weight 0, and the multiplier systems cancel. Its index is $\sum_j \frac12 (n_j^2-m_j^2)$. In general however the quotients has poles, unless the zeros of the denominators are cancelled by zeros in the numerator. For a single quotient for instance,  $\frac{\theta_1(\tau,kz)}{\theta_1(\tau,k'z)}$ is holomorphic so long as $\{ z = \frac{1}{k'}(\lambda \tau +\mu) \mid \lambda, \mu \in \mathbb{Z} \} \subset \{z = \frac{1}{k}(\lambda \tau + \mu) \mid \lambda, \mu \in \mathbb{Z} \},$ i.e. $k' \mid k$. To obtain genuine weak Jacobi forms it is thus necessary have similar cancellations.

\subsubsection{Single quotient}
Let us first discuss under what conditions single theta quotient
\begin{align}  \label{eq:singlethetaquotient}
\varphi_{0,\frac{1}{2}(\alpha^2-\beta^2)}(\tau,z)=\frac{\theta_1(\tau,\alpha z)}{\theta_1(\tau,\beta z)}
\end{align}
have slow growth about the most polar term $y^{b}=y^{\frac{1}{2}(\alpha-\beta)}$. The result is given in the following proposition:
\begin{proposition} \label{prop:ClassificationSingleTheta}
	The single theta quotients that have slow growing $f_{0,b}(n,l)$ about their most polar term $y^b$ are given by quotients of the form
	\begin{align}
	\frac{\theta_1(\tau,(k+1)\beta z)}{\theta_1(\tau,\beta z)},
	\end{align}
	for $k$ even or $\beta$ even. For such a quotient, the most polar term is $y^{k\beta/2}$ and the index is $\frac{\beta^2 k(k+2)}{2}$.
\end{proposition}
\begin{proof}
	For the quotient \eqref{eq:singlethetaquotient} to be holomorphic on $\mathbb{H} \times \mathbb{C}$, we must have $\beta \mid \alpha$, so we may write $\alpha = (k+1)\beta$. To obtain an integral weight $t=\frac{\beta^2(k^2+2k)}{2}$, we must have that $\beta$ is even or $k$ is even. 
	
	Next let us discuss the regularity at $q=0$ of the most general single theta quotient (\ref{eq:singlethetaquotient}). From Proposition \ref{prop:GeneratingFunctions}, we know slow growth about $y^{b}$ is equivalent to regularity at $q=0$ of the collection of modular forms $\chi_{r,s}(\tau)$, $0 \leq r,s \leq b-1$, coming from specializations of the weak Jacobi form $\varphi_{0,\frac{1}{2}(\alpha^2-\beta^2)}$.

The specializations $\chi_{r,s}(\tau)$ of the quotient \eqref{eq:singlethetaquotient} are
\begin{align} \label{eq:SpecializedSingleThetaQuotient}
\begin{split}
\chi_{r,s}(\tau)&=\frac{q^{\frac{\alpha^2}{2}\frac{r^2}{b^2}}\theta_1(\tau,\alpha(\frac{r}{b} \tau + \frac{s}{b}))}{q^{\frac{\beta^2}{2}\frac{r^2}{b^2}}\theta_1(\tau,\beta(\frac{r}{b} \tau + \frac{s}{b}))}\\  
&=\frac{-q^{\frac{\alpha^2}{2}\frac{r^2}{b^2}}q^{1/8}q^{-\frac{\alpha}{2}\frac{r}{b}} e^{-2 \pi i \frac{\alpha}{2}\frac{s}{b}}}{-q^{\frac{\beta^2}{2}\frac{r^2}{b^2}}q^{1/8}q^{-\frac{\beta}{2}\frac{r}{b}} e^{-2 \pi i \frac{\beta}{2}\frac{s}{b}}  } \\ & \hspace*{1.5cm} \times
\frac{\overset{\infty}{\underset{n=1}{\prod}}(1-q^n)(1-q^{n-1}q^{\alpha \frac{r}{b}}e^{2 \pi i \alpha \frac{s}{b}})(1-q^n q^{-\alpha\frac{r}{b}}e^{-2 \pi i \alpha \frac{s}{b}})}{\overset{\infty}{\underset{n=1}{\prod}}(1-q^n)(1-q^{n-1}q^{\beta \frac{r}{b}}e^{2 \pi i \beta \frac{s}{b}})(1-q^n q^{-\beta\frac{r}{b}}e^{-2 \pi i \beta \frac{s}{b}})}.
\end{split}
\end{align}

Regularity of the specialized theta quotient \eqref{eq:SpecializedSingleThetaQuotient} at $q=0$ is equivalent to it having only nonnegative powers of $q$ in its Fourier expansion. Thus, we only need compare the lowest powers of $q$ in the numerator and denominator: the lowest power of $q$ in the numerator of $\chi_{r,s}(\tau)$ is greater than or equal to the power of $q$ in the denominator if and only if $\chi_{r,s}(\tau)$ is regular at $q=0$. Note for $\chi_{0,s}(\tau)$, this approach does not work as naive computation leads to an undefined quotient. However, this case is easily managed as the Fourier-Jacobi expansion of the form $\varphi_{0,\frac{1}{2}(\alpha^2-\beta^2)}(\tau,z)$ is nonnegative in $q$ and for $\chi_{0,s}(\tau)$, its variable $y$ is specialized to $e^{2\pi i \frac{s}{b}}$ which does not modify the powers of $q$. So $\chi_{0,s}(\tau,z)$ is always regular at $q=0$. 

The term with the lowest power of $q$ in $q^{\frac{\kappa^2}{2}\frac{r^2}{b^2}}\theta_1(\tau,\kappa(\frac{r}{b} \tau+\frac{s}{b}))$ is given by multiplying out all $q^{n-\kappa \frac{r}{b}}$ with negative  $n-\kappa \frac{r}{b}$ in the third factor of the product formula \eqref{eq:GritsenkoTheta}. The lowest power of $q$ in the specialization \eqref{eq:SpecializedSingleThetaQuotient} is then
\begin{align} \label{eq:LowestPowerThetaQuotient}
\frac{q^{\frac{\alpha^2}{2} \frac{r^2}{b^2}+\frac{1}{8} - \frac{\alpha}{2}\frac{r}{b} +\overset{\lfloor \alpha \frac{r}{b} \rfloor}{\underset{n=1}{\sum}} n- \alpha \frac{r}{b}}}{q^{\frac{\beta^2}{2} \frac{r^2}{b^2}+\frac{1}{8} - \frac{\beta}{2}\frac{r}{b} +\overset{\lfloor \beta \frac{r}{b} \rfloor}{\underset{n=1}{\sum}} n- \beta \frac{r}{b}}} =q^{\frac{\alpha^2-\beta^2}{2} \frac{r^2}{b^2} - \frac{\alpha-\beta}{2}\frac{r}{b} +\frac{\lfloor \alpha \frac{r}{b} \rfloor(\lfloor \alpha \frac{r}{b} \rfloor +1)}{2}-\frac{\lfloor \beta \frac{r}{b} \rfloor(\lfloor \beta \frac{r}{b} \rfloor +1)}{2}- \alpha \frac{r}{b} \lfloor \alpha \frac{r}{b} \rfloor+\beta\frac{r}{b} \lfloor \beta \frac{r}{b} \rfloor}.
\end{align}

The condition for the theta quotient \eqref{eq:singlethetaquotient} to have slow growing $f_{0,b}(n,l)$, for $b=\frac{1}{2}(\alpha-\beta)$, is then
\begin{align} \label{eq:ConditionSlowGrowthSingleQuotient}
\begin{split}
\frac{\alpha^2-\beta^2}{2} \frac{r^2}{b^2} - \frac{\alpha-\beta}{2}\frac{r}{b} +\frac{\lfloor \alpha \frac{r}{b} \rfloor(\lfloor \alpha \frac{r}{b} \rfloor +1)}{2}-\frac{\lfloor \beta \frac{r}{b} \rfloor(\lfloor \beta \frac{r}{b} \rfloor +1)}{2} \\ - \alpha \frac{r}{b} \lfloor \alpha \frac{r}{b} \rfloor+\beta\frac{r}{b} \lfloor \beta \frac{r}{b} \rfloor \geq 0, \ \ \ 1 \leq r \leq b-1.
\end{split}
\end{align}

	Slow growth is equivalent to the condition \eqref{eq:ConditionSlowGrowthSingleQuotient}. We have $b=\frac{k\beta}{2}$ and the left side of \eqref{eq:ConditionSlowGrowthSingleQuotient} may be written as 
	\begin{align} \label{eq:ProofofSingleClassification}
	\begin{split}
	2\frac{k+2}{k}r^2-r+\frac{ \lfloor (k+1) \frac{2r}{k} \rfloor ( \lfloor (k+1) \frac{2 r}{k} \rfloor+1)}{2}-\frac{ \lfloor \frac{2r}{k} \rfloor ( \lfloor \frac{2 r}{k} \rfloor+1)}{2} \\ -(k+1)\frac{2r}{k} \lfloor (k+1) \frac{2r}{k} \rfloor + \frac{2r}{k} \lfloor \frac{2r}{k} \rfloor, \ \ \  0< r < \frac{k\beta}{2}.
	\end{split}
	\end{align}
	Using $\lfloor \frac{k+1}{k}2r \rfloor = \lfloor 2r + \frac{2r}{k} \rfloor = 2r + \lfloor \frac{2r}{k} \rfloor$, \eqref{eq:ProofofSingleClassification} reduces to 0 for each value of $r$. Thus, we have slow growth.
\end{proof}

\subsubsection{Multiple theta quotients.} We consider the general case of multiple theta quotients. Unlike for single quotients, the situation is more complicated, so that we are not able to give a complete classification of all slow growing quotients. Instead, we numerically compute a list up to index 39.

The condition for the theta quotient 
\begin{align}
\overset{N}{\underset{j=1}{\prod}} \frac{\theta_1(\tau, n_1 z)\theta_1(\tau,n_2 z) \cdots \theta_1(\tau,n_N z)}{\theta_1(\tau, m_1 z)\theta_1(\tau,m_2 z) \cdots \theta_1(\tau,m_N z)}
\end{align}
to be slow growth about $y^{b}=y^{\frac{1}{2}(\overset{N}{\underset{j=1}{\sum}} n_j - m_j)}$
is just the sum of \eqref{eq:ConditionSlowGrowthSingleQuotient} over each quotient, leading to the condition
\begin{align} \label{eq:ConditionSlowGrowthMultiQuotient}
\begin{split}
\overset{N}{\underset{j=1}{\sum}} \frac{n_j^2-m_j^2}{2} \frac{r^2}{b^2} - \frac{n_j-m_j}{2}\frac{r}{b} +\frac{\lfloor n_j \frac{r}{b} \rfloor(\lfloor n_j \frac{r}{b} \rfloor +1)}{2}-\frac{\lfloor m_j \frac{r}{b} \rfloor(\lfloor m_j \frac{r}{b} \rfloor +1)}{2} \\ - n_j \frac{r}{b} \lfloor n_j \frac{r}{b} \rfloor+m_j\frac{r}{b} \lfloor m_j \frac{r}{b} \rfloor \geq 0, \ \ \ 1 \leq r \leq b-1.
\end{split}
\end{align}

Unlike the case of single theta quotients, the most polar term is not guaranteed to be $y^b$, indeed some $q^a y^b$ for $a >0$ may be the most polar term. 

We computed all theta quotients up to $N=7$ quotients for index $1 \leq m \leq 61$ and checked them for slow growth about $y^b$ using the condition \eqref{eq:ConditionSlowGrowthMultiQuotient}. For each index $m$ and $b=\frac{1}{2}\left(\overset{N}{\underset{j=1}{\sum}} n_j - m_j\right)$, we found the dimension of the space $\Js_\theta$ spanned by theta quotients that have slow growth at $y^b$. We present our results in (\ref{table:ThetaQuotients}), and we include the corresponding dimension of $\Js_{m}^{0,b}$. Note that the two dimensions presented are not directly comparable, since a theta quotient may not have $y^b$ as its most polar term.

\begin{table}[H]
	\centering
	\begin{tabular}{ c  c  c c | c c c c  }
	m & b & dim $\Js_m^{0,b}$ & $\dim \Js_\theta$ & 	m & b & dim $\Js_m^{0,b}$ & $\dim \Js_\theta$ \\
3 & 1 & 1 & 1 &22 & 3 & 1 & 1 \\
4&1&1&1 & 22 & 4 & 2 & 2 \\
6&1 & 1 & 1 & 24 & 2 & 1 & 2\\
6&2&2&1 &24 & 3 & 2 & 3\\
7&2&1&1 & 24 &4 & 4 & 4\\
8&2&2&1 & 25 & 4 & 1 & 1\\
9&2&1&1 & 25 & 5 & 4 & 2\\
9&3&3&1 & 26 & 4 & 2 & 1\\
10 & 2 & 1 & 1& 26 & 5 & 2 & 1\\
10 & 3 & 2 & 1 & 27 & 3 & 1 & 2\\
11 & 3 & 1 & 1 & 27 & 4 & 1 & 1\\
12 & 2 &2 &2 & 27 & 5 & 2 & 2\\
12 & 3 & 3 & 2& 28 & 3 & 1 & 2\\
13 &3 & 1 & 1 & 28 & 4 & 3 & 3\\
14 &3 &1 & 1 & 30 & 3 & 1 & 2\\
15 & 2 & 1 & 1 & 30 & 4 & 2 & 3\\
15 & 3 & 2 & 2  & 30 & 5 & 4 & 4\\
16 & 2 & 1 & 1 & 31 & 5 & 1 & 1 \\
16 &3 & 2 & 2  & 32 & 4 & 3 & 2\\
16 & 4 & 4 & 2 & 32 & 5 & 2 & 1\\
17 & 4 & 2 & 1 & 33 & 4 & 2 & 2\\
18 & 2 & 0 & 1& 33 & 5 & 1  & 1\\
18 &3 & 3 & 3 & 34 & 4 & 1 & 2\\
18 &4 & 3 & 2 & 34 & 5 & 1 & 1\\
19 & 3 & 1 & 1 & 35 & 5 & 3 & 1\\
19 & 4 & 1 & 1 & 36 & 3 & 1 & 3\\
20 & 3 & 1 & 1 & 36 & 4 & 3 & 4\\
20 & 4 & 4 & 2 & 37 & 6 & 3 & 3\\
21 & 3 & 1 & 1 & 38 & 6 & 3 & 3\\
21 & 4 & 2 & 2 & 39 & 3 & 0 & 1
\end{tabular}
\caption{The dimension of the space of slow growing theta quotients $\Js_\theta$ about a fixed $y^b$ and index $m$. For comparison we also give the dimension of all slow growing weak Jacobi forms.}
\label{table:ThetaQuotients}
\end{table}

Let us finish this section by comparing to and extending the results obtained in \cite{Belin:2020nmp}.
Infinite classes of slow growing wJfs were obtained in \cite{Belin:2020nmp} from the elliptic genera of  ADE minimal models. As a corollary of Proposition~\ref{prop:ClassificationSingleTheta} and \eqref{eq:ConditionSlowGrowthMultiQuotient}, we can complete and greatly streamline  their proof:
\begin{corollary}\label{cor:ADEmm}
The ADE minimal model wJfs given in (3.8) of \cite{Belin:2020nmp} are all slow growing around $q^0 y^b$ of maximal polarity.
\end{corollary}
\begin{proof}
The A series are single theta quotients and so are a special case of Proposition \ref{prop:ClassificationSingleTheta}. The three E series may each be checked individually, using \eqref{eq:ConditionSlowGrowthMultiQuotient}. For the D series, it is easy to see that \eqref{eq:ConditionSlowGrowthMultiQuotient} is equal to zero for all $1 \leq r \leq b-1$.
\end{proof}

Furthermore \cite{Belin:2020nmp} proposed a conjecture that the wJfs obtained from the elliptic genus of the $A_{k+1}$ Kazama-Suzuki models with $M=2$ are slow growing. We give a simple proof of this conjecture using the criterion \eqref{eq:ConditionSlowGrowthMultiQuotient}:

\begin{proposition} \label{lemma:KSmodel}
	The weak Jacobi forms for the $M=2$ $A_{k+1}$ Kazama-Suzuki models, defined as
	\begin{align} \label{eq:KazamaSuzuiki}
	\varphi^{2,k}(\tau,z)=\frac{\theta_1(\tau,(k+1)z)}{\theta_1(\tau,z)} \frac{\theta_1(\tau,(k+2)z)}{\theta_1(\tau,2z)}, \ \  k \in \mathbb{Z}_{>0}
	\end{align}
	have slow growth about the polar term $y^k$ of maximal polarity.
\end{proposition}
\begin{proof}
From \eqref{eq:ConditionSlowGrowthMultiQuotient}, $\varphi^{2,k}$ has slow growth about $y^k$ if and only if for all $ 1 \leq r \leq k-1$, the following expression is greater than or equal to 0:
	 \begin{multline} \label{eq:KSModelComputation}
	 \frac{(k+1)^2-1}{2}\frac{r^2}{k^2}-\frac{(k+1)-1}{2}\frac{r}{k}
	 +\frac{(k+2)^2-2^2}{2}\frac{r^2}{k^2}-\frac{k+2-2}{2}\frac{r}{k} \\
	 +\frac{ \lfloor (k+1) \frac{r}{k} \rfloor (\lfloor (k+1) \frac{r}{k} \rfloor +1)}{2} -\frac{ \lfloor \frac{r}{k} \rfloor (\lfloor \frac{r}{k} \rfloor +1)}{2} -(k+1)\frac{r}{k} \lfloor (k+1) \frac{r}{k} \rfloor + \frac{r}{k} \lfloor \frac{r}{k} \rfloor \\ + \frac{\lfloor (k+2)\frac{r}{k}\rfloor (\lfloor (k+2) \frac{r}{k} \rfloor +1)}{2}  -\frac{ \lfloor 2 \frac{r}{k} \rfloor (\lfloor 2 \frac{r}{k} \rfloor +1)}{2}-(k+2) \frac{r}{k} \lfloor (k+2) \frac{r}{k} \rfloor + 2\frac{r}{k} \lfloor 2 \frac{r}{k} \rfloor.
	\end{multline}
	 To simplify the expression, we have $0<\frac{r}{k} <1$ for each $r$ so that
	 \begin{align}
	 \lfloor (k+1) \frac{r}{k} \rfloor = r, \ \ \ \ \ \ \ \lfloor (k+2) \frac{r}{k} \rfloor = r+\lfloor 2 \frac{r}{k} \rfloor.
	 \end{align}
	 This reduces \eqref{eq:KSModelComputation} to
 	 \begin{multline}
	 \frac{k^2+2k}{2}\frac{r^2}{k^2}-\frac{r}{2}+\frac{k^2+4k}{2}\frac{r^2}{k^2}-\frac{r}{2}+\frac{r(r+1)}{2}-(k+1)\frac{r^2}{k} \\ +\frac{(r+\lfloor 2 \frac{r}{k} \rfloor)(r+\lfloor 2 \frac{r}{k} \rfloor +1)}{2}-\frac{\lfloor 2 \frac{r}{k} \rfloor (\lfloor 2 \frac{r}{k} \rfloor +1)}{2}-(k+2)\frac{r}{k}(r+\lfloor 2 \frac{r}{k} \rfloor)+2\frac{r}{k} \lfloor 2 \frac{r}{k} \rfloor.
	 \end{multline}
	 Now, a simple matter of cancellations gives us that the above expression equals 0 for each value of $r$.
\end{proof}

\section{Growth around $q^a y^b$}\label{s:qayb}
Let us finally say a few words about slow growth about $q^a y^b$ with $a>0$. In this case, there is no analogue to proposition~\ref{prop:GeneratingFunctions}. This means that even for a given wJf, it is not immediately obvious how to determine if it is slow growing. In this section we gather evidence that the growth behaviors of $f_{a,b}(n,l)$ are nonetheless characterized in the same way as $f_{0,b}(n,l)$: \begin{itemize}
	\item[1.] $f_{a,b}(n,l)$ either has exponential growth in $n$ and $l$ or is bounded as a function of $n$ and $l$.
 \item[2.] 

  When  $f_{a,b}(n,l)$ is bounded as a function of $n$ and $l$,  there are integers $e, f, g, h \in \mathbb{Z}$ such that
\begin{align*}
f_{a,b}(n,l) = \begin{cases} \text{ nonzero } & : e n + f l =0 \text{ or } g n + h l =0 \\0 & : \text{ else}. \end{cases}
\end{align*}
\end{itemize}
Even though we mainly gather numerical evidence, let us start out with some rigorous analytical results for index $m=6$.

\subsection{Analytical Results}
For $m=6$, table~\ref{table:ybTable} shows that the space $\Js_6^{0,1}$ of slow growing forms about $q^0y^1$ is generated by a single wJf, which can be written as the theta quotient $\phi_6 :=\frac{\theta_1(\tau,4z)}{\theta_1(\tau,2z)}$. $\phi_6$ has two terms of polarity 1, namely $q^0y$ and $q^1 y^5$. We will establish that $\phi_6$ is also slow growing about $q^1y^5$, and that the $f_{1,5}$ are in fact closely related to the $f_{1,0}$.
For this, let us first introduce
\begin{definition} Let $\varphi$ be a wJf of index $m$. Let $\sigma$ be a permutation of $\mathbb{Z}/2m \mathbb{Z}$. We define the operation $\hat{W}_\sigma$ to be
\begin{align*}
\hat{W}_\sigma : \varphi= \underset{l \text{ mod } 2m}{\sum} h_l(\tau) \theta_{m,l}(\tau,z) \mapsto \underset{l \text{ mod } 2m}{\sum} h_{\sigma(l)}(\tau)\theta_{m,l}(\tau,z).
\end{align*}
\end{definition}
This $\hat W_\sigma$ is a generalization of the Atkin-Lehner involution $W_n$, for which $\sigma$ is given by a multiplication by some integer $\xi$, $\sigma(l)=\xi l$. Atkin-Lehner involutions $W_n$ map Jacobi forms to Jacobi forms \cite[Thm~5.2]{MR781735}, but not necessarily weak Jacobi forms to weak Jacobi forms, as they may introduce negative powers of $q$, giving nearly holomorphic Jacobi forms.
For general $\sigma$, our generalized $\hat W_\sigma$ may not even give any Jacobi type form. However, as we will see, it can still be a useful tool to extract coefficients more easily.

\begin{proposition}\label{prop:W3f15}
Take $\varphi\in J_{0,6}$. Then 
\be
f_{1,5}(n,l) = \hat f_{0,1}(2n+l,-9n-5l)
\ee
where $\hat f$ are the $f$ of the weak Jacobi form $\hat W_\sigma\varphi$ with $\sigma=(1 \ 5)(2\ 10)(4\ 8)(7\  11)$.
\end{proposition}
Note that this $\hat W_\sigma$ is actually $W_3$, the Atkins-Lehner involution with $\xi=5$.
\begin{proof}
Consider the theta decomposition
$\varphi(\tau,z)= \underset{\mu \text{ mod } 2m}{\sum} h_{\mu}(\tau) \theta_{m,\mu}(\tau,z)$. As is well known, the coefficients $c(n,l)$ of $\varphi$ depend only on $l \text{ mod } 2m\mathbb{Z}$ and the discriminant $4mn-l^2$. 
By definition,
\begin{align}
    f_{1,5}(n,l)=\underset{r \in \mathbb{Z}}{\sum} c(nr+r^2,l-5r) \ .
\end{align}
One can verify that the coefficient $c(nr+r^2,l-5r)$ has the same discriminant as the coefficient $c\left((2n+l)(r+3n+2l),-9n-5l-(r+3n+2l)\right)$ and, over the ring $\mathbb{Z}/12\mathbb{Z}$, the second argument of the former coefficient is five times the second argument of the latter coefficient. 
This means that for  $c\left((2n+l)(r+3n+2l),-9n-5l-(r+3n+2l)\right)$ appearing as the coefficient of some $q^Ny^{\mu+2mk}$ in $\varphi$, $c(nr+r^2,l-5r)$ appears as the coefficient of $q^Ny^{5\mu+2mk}$. Note that the map $\mu \mapsto 5\mu$ in the ring $\mathbb{Z}/12\mathbb{Z}$ applies precisely the permutation $\sigma$ to $\mu$, so that $c(nr+r^2,l-5r)$ is the coefficient of $q^Ny^{\mu+2mk}$ in $\hat{W}_\sigma(\varphi)$.

\end{proof}

\begin{corollary} \label{cor:Index6SameSpaces} $\Js^{1,5}_6=\Js^{0,1}_6$.  It is spanned by $\phi_6:= \frac{\theta_1(\tau,4z)}{\theta_1(\tau,2z)}$, whose  $f_{1,5}(n,l)$ and $f_{0,1}(n,l)$ are given by
	\begin{align}\label{f01}
	f_{0,1}(n,l) = \begin{cases} 2 & n=0 \text{ or } 6n+l=0 \\ 0 & \text{ else }, \end{cases}
	\end{align} 
    and
	\begin{align}\label{f15}
	f_{1,5}(n,l) = \begin{cases} -2 & 2n+l=0 \text{ or } 3n+l=0 \\ 0 & \text{ else }. \end{cases}
	\end{align} 
\end{corollary}
\begin{proof}
First note that (\ref{f01}) was established in \cite{Belin:2019jqz}. As mentioned above, $\Js^{0,1}_6$ is spanned by $\phi_6$.
To show that $\Js^{1,5}_6=\Js^{0,1}_6$, first note that up to multiplication, $\phi_6$ is the only wJf in $J_{0,6}$ that has no terms of polarity greater than 1, which establishes that $\Js^{1,5}_6\subset\Js^{0,1}_6$. To establish that $\phi_6$ is slow growing about $q^1y^5$, we will show that
\be\label{ALEigen}
 W_3\phi_6 = -\phi_6\ .
\ee
Proposition~\ref{prop:W3f15} then immediately implies (\ref{f15}), from which it follows that $\varphi$ is indeed slow growing about $q^1y^5$.
To prove (\ref{ALEigen}), we first establish that $ W_3\phi_6$ is a weak Jacobi form. This follows from the fact that the most polar term of $\phi_6$ has polarity 1, which means that $W_3\phi_6$ has powers of $q$ greater or equal to $-1/24$. Since it has integer powers, it follows that it cannot have any negative powers at all, thus establishing that it is a wJf.
To show the equality we then simply check that the polar terms of the two sides agree.
\end{proof}

In principle it is possible to apply this type of procedure also to other cases,  to relate $f_{a,b}$ to some $\hat f_{0,b'}$. The issue however is that the $\hat W_\sigma$  may not be an Atkin-Lehner involution. This means that $\hat W_\sigma \phi$ may not be a wJf, so that we cannot use a known expression for $\hat f_{0,b'}$. However, using this approach can still be useful for numerical computations: If $a>0$, evaluating the sum in (\ref{eq:sumsofcoeffs}) to compute $f_{a,b}(n,l)$ quickly becomes very expensive because $\phi$ has to be evaluated to high powers in $q$. It can thus be cheaper to instead evaluate $\hat f_{0,b'}$ for the function $\hat\phi:=\hat W_\sigma \phi$, even if $\hat\phi$ is not a wJf. An example for this is the following proposition:
\begin{proposition}\label{prop:f16}
Let $\varphi\in J_{0,8}$ and $\sigma=(2 \ 6)(4 \ 12)(10 \ 14)$. We then have 
\be
f_{1,6}(2n,2l) = \hat f_{0,2}(4n+2l,-12n-7l)
\ee
where  $f$ comes from $\varphi$, and $\hat f$ from $\hat{W}_\sigma \varphi$.
\end{proposition}
\begin{proof}
Analog to proposition~\ref{prop:W3f15}. Note that $\hat W_\sigma$ acts like an Atkin-Lehner with $\xi=3$, but only on the even powers.

\end{proof}

Let us note that proposition~\ref{prop:W3f15} and to a lesser extent also proposition~\ref{prop:f16} provide some evidence that the $f_{a,b}$ can be interpreted as the Fourier coefficients of some modular object, similar to the case of $f_{0,b}$ described in \cite{Belin:2019jqz}.

\subsection{Numerical Results}
Let us finally quote some purely experimental results for the dimension of slow growing forms about $q^ay^b$, $a>0$ for small index $m$. To numerically find the dimension of $\Js^{a,b}_m$, we computed  $f_{a,b}(n,l)$ numerically for some values of $n,l$ and checked its growth behavior. These computations become expensive very quickly because the Fourier-Jacobi coefficients in the sum of $f_{a,b}(n,l)$ are of very high order. We therefore had  to restrict to fairly small values of $n$ and $l$. However, we still believe that for the values given below, our determination of fast vs.\ slow growth is fairly reliable: When the $f_{a,b}(n,l)$ are not bounded, then already very small values of $n$ and $l$ give large numbers. When they are bounded, $f(n,l)$ is zero or very small for all $n$ and $l$ we tested.  We put our conjectured findings in table~\ref{table:qaybTable}. In each case, the affine space $\hat\Js^{a,b}_m$ is nonempty.

\begin{table}[H]
\centering
	\begin{tabular}{ c | c | c }
		m & term,polarity & dim $\Js^{a,b}_m$ \\ 
		\hline
		5 & $q^1y^5$,5 &  1 \\
		6 & $q^1y^5$,1 &  1 \\
		7 & $q^1y^6$,8 &  1 \\
		8 & $q^1y^6$,4 &  2 \\
		9 & $q^2y^9$,9 &  2 \\
		10 & $q^1y^7$,9 & 2 \\
		11 & $q^1y^7$,5 &  1\\
		11 & $q^2y^{10}$,12 &  2 \\
		12 & $q^2y^{10}$,4 & 2 \\
		12 & $q^1y^8$,16 & 2
	\end{tabular}
\caption{Conjectured dimension of slow growing forms, based on numerical analysis.}
\label{table:qaybTable}
\end{table}

We highlight some interesting observations. Index 7, 11, and 12 have wJfs slow growing around terms $q^a y^b$ of polarity larger than the index $m$. This is in contrast with the $q^0y^b$ case, where $\Js^{0,b}_m$ is zero for any $y^b$ with polarity larger than $m$. Additionally, we find wJfs with index 9 and 10 that are slow growing around $y^3$ but fast growing about $q^2y^9$ and $q^1 y^7$, respectively, of the same polarity. Similarly, for index 12, we find a wJf that is slow growing about $q^1y^8$, which has polarity 16, but fast growing around $y^4$, also of polarity 16. Thus, in contrast to the specific example of index 6 discussed in corollary \ref{cor:Index6SameSpaces}, it is not necessarily the case that fast (resp. slow) growth about one term implies fast (resp. slow) growth about another term of the same polarity.

\bibliographystyle{alpha}
 \bibliography{./main}

\end{document}